\documentclass{amsart}
\usepackage[utf8x]{inputenc}
\usepackage{amsthm}
\usepackage{amsmath}
\usepackage{amssymb}
\usepackage{amsfonts} 

\newtheorem{theorem}{Theorem}

\newtheorem{corollary}[theorem]{Corollary}

\newtheorem{definition}[theorem]{Definition}

\newtheorem{lemma}[theorem]{Lemma}

\newtheorem{proposition}[theorem]{Proposition}
\newtheorem{remark}[theorem]{Remark}

\newcommand{\R}{\mathbb{R}}

\newcommand{\Z}{\mathbb{Z}}
\newcommand{\Q}{\mathbb{Q}}

\newcommand{\T}{\mathbb{T}}
\newcommand{\pp}{\mathfrak{p}}
\newcommand{\qq}{\mathfrak{q}}
\newcommand{\Id}{\textup{Id}}

\begin{document}

\title[Local Conjugacy] {Local Conjugacy of Irreducible Hyperbolic Toral Automorphisms}
\author[Bakker]{Lennard F. Bakker}
\address{Department of Mathematics, Brigham Young University, Provo, Utah, USA}
\email{bakker@math.byu.edu}
\author[Martins Rodrigues]{Pedro Martins Rodrigues}
\address{Department of Mathematics, Instituto Superior T\'ecnico, Univ. Tec. Lisboa, Lisboa, Portugal}
\email{pmartins@math.ist.utl.pt}
\author[Moreira]{Miguel M. R. Moreira}
\address{Department of Mathematics, Instituto Superior T\'ecnico, Univ. Tec. Lisboa, Lisboa, Portugal}

\commby{}

\begin{abstract} This paper is dedicated to the conjugacy problem in $GL(n,\Z)$ and its connection with algebraic number theory. This connection may be summed up in the Latimer-MacDuffee-Taussky Theorem, which, in a very broad sense, identifies the conjugacy relation in $GL(n,\Z)$ with the relation of arithmetic equivalence between certain ideals of algebraic numbers. The main purpouse is to clarify the link between a weaker relation between ideals, weak equivalence, and local conjugacy in $GL(n,\Z)$, i.e., the conjugacy of $GL(n,\Z)$ matrices in the groups $GL(n,\Z_{p})$ of invertible matrices over the $p$-adic integers.
 \end{abstract}

\maketitle

\section{Introduction}

The conjugacy problem in $GL(n,\Z)$ has a classic and well known number theoretic counterpart in the form of the Latimer-MacDuffee-Taussky Theorem (see \cite{LMD33} and \cite{Tau49} for the original statements). We recall here this connection, giving also some of the basic definitions and notations to be used in the sequel.

Fix an irreducible polynomial $f(t) \in \Z[t]$ of degree $n$. In the quotient ring $\Z[t]/(f(t))$, the class of $t$, which we denote from now on by
$\beta$, is a simple root of $f$ and we may identify that ring with $\Z[\beta]$; we denote its field of fractions as $K$, which may also be identified with the isomorphic field $\Q[\beta]$. A matrix $A \in GL_{n}(\Z)$ with characteristic polynomial $f(t)$ acts on $K^{n}$ and so there exists a right eigenvector $u \in K^{n}$ associated to $\beta$: $A u=\beta u$; the entries $u_{i}$ of $u$ determine a $\Z[\beta]$ fractional ideal $I_{A}=\bigoplus_{i} u_{i}\Z$ and $A$ represents multiplication by $\beta$ in $I$, with respect to the basis $(u_{i})$.  We have

\begin{definition}Two fractional ideals $I$ and $J$ are arithmetically equivalent if $J=\alpha I$ for some $\alpha \in K$.
\end{definition}
 
So each $A \in GL_{n}(\Z)$ is in fact associated with a class of ideals. On the other hand, if $A$ and $B$ are conjugated they are associated with the same ideal class, the conjugating matrix $P$ acting as a change of basis of the ideals, taken as $\Z$-modules, thus establishing a bijection between the conjugacy classes of matrices with characteristic polynomial $f(t)$ and
the arithmetic equivalence classes of ideals in $R_{0}$.

This result is in fact true in a more general setting (see, for instance, \cite{Wan08} and \cite{Zil11}): replacing $\Z$ by any PID $D$, given $A \in M_{n}(D)$, we may define $D[\beta]$, its field  of fractions $K$, the associated ideal $I_{A}$, and the relation of arithmetic equivalence of ideals as above; we have then 

\begin{theorem}
Let $D$ be a PID and $A,B\in M_{n}(D)$ be two matrices with the same characteristic irreducible and separable polynomial $f$. Then $A$ and $B$ are similar over $D$, ie, $A P=P B$ for some $P \in GL(n,D)$, if and only if $I_A$ and $I_B$ are arithmetically equivalent over $D[\beta]$. 
\end{theorem}

In the case $D=\Z$, the ring $\Z[\beta]$ is a subring of $O_{K}$, the ring of integers of the algebraic number field $K$, and the determination of the arithmetic equivalence classes of ideals is a generalization of the problem of determining the ideal class group of the field, a classic problem about which many questions remain open, even for $n=2$. It is thus desirable to study larger equivalence relations of ideals and the respective relations of matrices. This is the case of weak equivalence of ideals, which may be defined for any integral domain: 

\begin{definition}
Let $R$ be an integral domain. The ideals $I$ and $J$ of $R$ are said to be weakly equivalent if there exist ideals $X$ and $Y$ such that $I=X J$ and $J=Y I$.  
\end{definition}

The next section gives a brief account of some results related with the definition of weak equivalence in the context of $\Z[\beta]$, following the original presentation in \cite{DTZ62}. For the moment it is enough to point out the following important theorem:

\begin{theorem}
Let $R$ be a noetherian domain. Two fractional ideals $I$ and $J$ are weakly equivalent if and only if their localizations $I_{P}$ and $J_{P}$ are arithmetically equivalent for any maximal ideal $P$.
\end{theorem}

The conjugacy problem in $GL(n,\Z)$ is also significant to dynamical system theory. Any $A \in GL(n,\Z)$ acts as an automorphism on various spaces, among which is the $n$-torus $\R^{n}/\Z^{n}$.  If the characteristic polynomial is irreducible and has no roots of absolute value $1$, the automorphism is hyperbolic and gives the prototypical example of an Anosov diffeomorphism, an important class of discrete dynamical systems with rich dynamical behaviour. Conjugacy of the matrices corresponds to topological conjugacy of the automorphisms, and so it is clearly important to find, and understand from  a dynamical point of view, invariants for this relation.

The set of periodic points of an automorphism - which coincides, in the hyperbolic case, with the rational torus $\Q^{n}/\Z^{n}$ - is a $\Z[\beta]$-module through the action of the automorphism.  The question of knowing to what extent the conjugacy class is determined by this algebraic structure was explored in previous papers (\cite{RS05}, \cite{BR12}). The most important conclusion found was that weak equivalence of the associated ideals of hyperbolic automorphisms $A$ and $B$ implies the existence of a bijection of the rational torus which is an isomorphism of $\Z[\beta]$-modules. 

This isomorphism induces an algebraic and topological isomorphism of the profinite completion $\overline{\Z^{n}}$ of $\Z^{n}$ which is the dual of the rational torus with the discrete topology.  Naturally this isomorphism is a conjugacy of the actions induced by $A$ and $B$ on $\overline{\Z^{n}}$. Because $\overline{\Z^{n}}\simeq \prod_{p} \Z_{p}^{n}$, where $\Z_{p}^{n}$ denotes the free $n$-dimensional module over the $p$-adic integers, the existence of this conjugacy is equivalent to $A$ and $B$ being conjugate over $\Z_{p}$ for all primes $p$, which has a clear bearing with the theorem quoted above.

The main purpose of this paper is to elaborate in this connection, establishing in particular the following

\begin{corollary}
Two matrices $A,B$ are similar over $\Z_{p}$, for all primes $p$, if and only if the ideals $I_{A}$ and $I_{B}$ are weakly equivalent.
\end{corollary}

Although the presentation relies on the language and concepts of algebraic number theory, the results needed are quite modest and stay almost entirely at an elementary level. References \cite{Neu99} or \cite{Mil14} provide ample background on this matter.

\section{Weak Equivalence of Ideals}

This section is devoted to a brief presentation of the relation of weak equivalence of ideals and of some of its properties and relations  with the conjugacy problem in $GL_{n}(\Z)$. A detailed
account of weak equivalence may be found in \cite{DTZ62}.

For two fractional ideals $I$ and $J$ of $\Z[\beta]$, we let $(J:I)$ denote the fractional ideal
\[(J:I)=\{x \in K: xI \subset J\}.\]
The coefficient ring of $I$ is defined to be $(I:I)$. It is an order of $K$, ie, a subring of $O_{K}$, the ring of integers of $K$, of rank $n$ over $\Z$, or, in other words, a subring of finite index of $O_{K}$. An order of $K$ is always a noetherian domain where all non-zero prime ideals are maximal but, with the obvious exception of $O_{K}$, is not integrally closed.

By definition, $\Z[\beta] \subset (I:I)$. The set of orders of $K$ containing $\Z[\beta]$ constitutes a finite lattice under inclusion. 

As two arithmetically equivalent $\Z[\beta]$ ideals have necessarily the same coefficient ring, this lattice provides a first step in the classification of ideals up to arithmetic equivalence and so of $GL_{n}(\Z)$ matrices up to conjugacy:  Suppose that $A \in GL_{n}(\Z)$ represents, as above, multiplication by $\beta$ in $I$, with respect to a basis over $\Z$; then, if $\theta \in K$ is given by a polynomial $\theta=p(\beta)$, then $\theta \in (I:I)$ if and only if $p(A)$ is an integer matrix. 

The orders of $K$ containing $\Z[\beta]$ are the idempotents of the multiplicative semigroup of $\Z[\beta]$ fractional ideals. Each order $R$ has an associated group of invertible ideals:

\begin{definition} A ideal $I$ is $R$-invertible if there exists an ideal $J$ such that $I J=R$.
\end{definition}

The ideal $J$ (the inverse of $I$) is $(R:I)$. It is not difficult to prove that, given an order $R$, an ideal $I$ is invertible for $R$ if and only if $R=(I:I)$ and $(R:(R:I))=I$.

It's well known that all $O_{K}$ ideals are invertible; however, for other orders $R$, there may exist ideals with ring of coefficients $R$ that are not invertible.

We come now to the main definition of this section:

\begin{definition} Two ideals $I$ and $J$ of $\Z[\beta]$ are weakly equivalent if they satisfy the following equivalent conditions:
\begin{enumerate}
 \item [1] There exist ideals $X$ and $Y$ such that $I=X J$ and $J=Y I$.
 \item [2] $(I:I)=(J:J)=R$ and there exists a $R$ invertible ideal $X$ such that $J=I X$.
 \item [3] $1 \in(I:J)(J:I)$.
\end{enumerate}
\end{definition}
 
A few facts are immediate consequences of the definition: arithmetically equivalent ideals are always weakly equivalent, so the weak equivalence relation contains arithmetic equivalence. On the other hand, two invertible but not arithmetically equivalent ideals, in an order $R$, are always weakly equivalent.
Furthermore, an invertible ideal and a non-invertible ideal, both in the same order $R$, are not weakly equivalent.

In \cite{DTZ62} the weak equivalence relation is defined in a slightly different way and in the general context of an abelian semigroup. The authors then prove that in the semigroup of ideals of a noetherian
domain, that relation is equivalent to the one defined above.

\subsection{Localization}

In this section we recall some results about localization will be used to clarify the relations between conjugacy of matrices and weak equivalence of the associated ideals in later sections.

\begin{definition}
Let $R$ be an integral domain and $\pp$ be a prime ideal of $R$. Then the localization of $R$ over $\pp$ is
\[R_\pp=\{a/b| a\in R, b\in R\setminus \pp\}\subseteq K.\]
If $I$ is a fractional ideal of $R$, define $I_\pp=IR_\pp$.
\end{definition}

Localization plays a fundamental role in Algebraic Number Theory. It behaves nicely with respect to the ideal operations and we have $(I+J)_\pp=I_\pp+J_\pp$ and similar properties for multiplication, intersection and quotient of ideals. Localization of ideals essentially eliminates from the ideal the information not respecting $\pp$. However, in some sense, from gathering the information from every prime ideal $\pp$ we can recover the information about the ideal.

\begin{proposition}
Let $I, J$ be fractional ideals of $R$. Then $I=J$ if and only if $I_\pp=J_\pp$ for every maximal ideal $\pp$.
\end{proposition}

For the proof of this and the following results in this section we refer to \cite{Neu99}.

The ring $R_\pp$ is a local ring, that is, it has only one maximal ideal, namely $\pp_\pp=\pp R_\pp$. Local rings have the nice property that their invertible ideals are principal.

\begin{proposition}
Let $S$ be a semi-local ring (that is, a ring with finitely many maximal ideals) and $I$ a fractional ideal of $S$. Then $I$ is invertible in $S$ if and only if it's principal. Therefore in a local ring ideals are arithmetically equivalent if and only if they are weakly equivalent.
\label{semilocal}
\end{proposition}

The great interest of this property is that it will show us that weak equivalence is actually a local version of arithmetic equivalence. 

\begin{proposition}
Let $R$ be an order and $I$ an ideal. Then $I$ is invertible in $R$ if and only if $I_\pp$ is principal for every prime ideal $\pp$ of $R$. Therefore $I$ and $J$ are weakly equivalent if and only if $I_\pp$ is arithmetically equivalent to $J_\pp$ for every prime ideal $\pp$ of $R$.
\label{localarithmetic}
\end{proposition}

\section{Local conjugacy in $GL(n,\Z)$ and weak equivalence of ideals}

We consider now conjugacy of matrices $A, B \in GL(n,\Z)$ over local rings and its connection to weak equivalence of ideals. As in the previous section, $\Z_{(p)}$ denotes the localization of $\Z$ at
the prime ideal $(p)$, while $\Z_{p}$ denotes, as usual, the ring of $p$-adic integers. We'll use also the notation $\Z_{/q}$ for the finite ring of congruence classes of integers modulo $q$; the notation
$a \equiv_{q} b$ for the congruence relation will also be used, with the obvious meaning, for vectors $a,b \in \Z^{m}$ or for integer matrices.

We start by establishing the fact that the similarity relations in $M(n,\Z)$ over these different local rings are in fact all the same. 

\begin{lemma} \label{modkernel}
Given $L\in M(m,\Z)$ and a prime $p$, let $D=diag[p^{k_{1}},\cdots, p^{k_{r}},0,\cdots,0]$ with $0\leq k_{1}\leq \cdots \leq k_{r}=\mu$, be the $p$ factor of the Smith Normal Form of $L$. Then, if  $\lambda$ is a nonnegative integer and $x'$ an integer vector such that $Lx' \equiv 0\, \mod p^{\mu+\lambda}$, there exists an integer vector $x$ such that 
 \[Lx=0 \, \mbox{ and }x \equiv x' \, \mod p^{\lambda}.\] 
\end{lemma}

\begin{proof} The conditions imply that $L=S D T$ where $S$ and $T$ are integer matrices with determinant prime to $p$. We have
 \[L x=0 \Leftrightarrow D Tx=0 \Leftrightarrow (T x)_{j}=0 \, \forall j\leq r.\]
 On the other hand, \[Tx'=y \mbox{ with } y_{j} \equiv 0\, \mod p^{\lambda}\,\, \forall j\leq r,\] since \[D Tx'=(p^{k_{1}}y_{1},\cdots,p^{k_{r}}y_{r},y_{r+1},\cdots,y_{m})^{t}\]
 must be $0$ modulo $p^{\mu+\lambda}$.\\
 \noindent We choose any $w \in \Z^{m}$ satisfying $w_{j}=0$ for all $j\leq r$ and $ (\mbox{det} T)w_{i}\equiv y_{i}$ for all $r<i\leq m$, and put $x=(\mbox{det} T)T^{-1}w$. Obviously $Lx=0$ and, because 
 $T(x-x')=(\mbox{det} T)w-y \equiv 0 \, \mod p^{\lambda}$ and $T$ is invertible over $\Z_{/p^{\lambda}}$, we have $x \equiv x'\, \mod p^{\lambda}$.\qedhere 
 \end{proof}

 The same reasoning proves also 
 
\begin{lemma} \label{padickernel}
Let $L\in M(m,\Z)$ and suppose that the system $Lx=0$ admits a solution with $x\in \Z_{p}^{m}$. Then the system admits a solution $y\in \Z^m$ such that $x\equiv_p y$.
\end{lemma}


For the next proposition, given a fixed pair of matrices $A,B\in M(n,\Z)$ and a prime $p$, we take $L \in M(n^{2},\Z)$ to be a matrix representing the linear map from $M(n,\Z)$ to itself defined by $L(X)=A X-X B$, and $\mu$ as in Lemma \ref{modkernel}.
\begin{proposition}
Given two matrices $A,B\in M(n,\Z)$ and a prime $p$, the following are equivalent.
\begin{enumerate}
\item $A$ and $B$ are similar over $\Z_p$.
\item There is a matrix $Q\in M_n(\Z)$ such that $AQ=QB$ and $p\nmid \det Q$.
\item $A$ and $B$ are similar over $\Z_{(p)}$.
\item $A$ and $B$ are similar over $\Z_{/p^{k}}$ for every $k\in \Z^{+}$. 
\item $A$ and $B$ are similar over $\Z_{/p^{\mu+1}}$.
\end{enumerate}
\label{equiv}
\end{proposition}

\begin{proof}
Suppose that $(1)$ holds. Consider the system of $n^2$ linear equations in the $n^2$ variables $s_{ij} \textup{ with } i,j\in \{1,2,\dots,n\}$ given by equating the entries of $AS-SB$ to $0$. Since $A$ and $B$ are similar over $\Z_p$, there is a matrix $S\in GL(n,\Z_p)$ satisfying $AS-SB=0$, from which follows by the lemma that there is a matrix $Q\in M(n,\Z)$ such that $AQ-QB=0$ and $Q\equiv_p S$; for this matrix, $\det Q \equiv_p \det S$, which must be nonzero modulo $p$, since $S\in GL_n(\Z_p)$, and therefore $(2)$ holds.

If $(2)$ holds, consider $Q$ as a matrix over $\Z_{(p)}$. Since $p\nmid \det Q$, $\det Q$ is a unity in 
$\Z_{(p)}$, and so $Q\in GL(n,\Z_{(p)})$, from which $(3)$ follows.

If $AQ=QB$ with $Q\in GL(n,\Z_{(p)})$, consider the image of $Q$ under the canonical homomorphism $\Z_{(p)}\rightarrow \Z_{/p^{k}}$. By an argument similar to the above we have $Q\in GL(n,\Z_{/p^{k}})$, and so $A$ and $B$ are similar over $\Z_{/p^{k}}$, and so, in particular, over $\Z_{/p^{\mu+1}}$.

Finally, suppose that $A$ and $B$ are similar over $\Z_{/p^{\mu +1}}$: there exists $X' \in M(n,\Z)$, with determinant prime to $p$, such that $A X' \equiv X'B \, \mod p^{\mu+1}$. Lemma \ref{modkernel} implies that  $A X=X B$ for some $X \in M(n,\Z)$ with $X \equiv X' \,\mod p$. This implies that the determinant of $X$ is prime to $p$ and so $X \in GL(n,\Z_{p})$.  
\qedhere
\end{proof}

Lemma \ref{modkernel} and the last implication Proposition \ref{equiv} are contained, in a more general form, in \cite{AO83}. As to Lemma \ref{padickernel} and the other parts of Proposition \ref{equiv}, although we are unable to give a reference for the original statement, we believe they are known.

\begin{remark}
An immediate consequence of the proposition is that two matrices $A,B\in M(n,\Z)$ are conjugated over $\Z_p$ for every prime $p$ if and only if there exist matrices $Q, S \in M(n,\Z)$ such that
 \[A Q = Q B, \,\,  A S = S B \, \mbox{  and  } \gcd(\det Q, \det S)=1.\]
\end{remark}

The following Lemma will be needed for the proof of the main results:

\begin{lemma}
Let $R$ be an order of $K$ containing $\Z[\beta]$ and $I$ a $R$-fractional ideal. Then we have the identity
$$\bigcap_{\pp \supseteq (p)} I_{\pp}=\Z_{(p)}I$$
where the intersection is over the prime ideals $\pp$ containing $p$.
\end{lemma}

\begin{proof}
Suppose that $x=\frac{y}{m}\in \Z_{(p)}I$ with $m\in \Z$ and $p\nmid m$. Since $\pp \cap \Z=(p)$ and $m\notin (p)$ we conclude that $m\notin \pp$, and therefore clearly $x\in I_{\pp}$ for every such $\pp$,
proving that $\Z_{(p)}I\subseteq \bigcap_{\pp \supseteq (p)} I_{\pp}$.

Now let the ideals $\pp \supseteq (p)$ be $\pp_{1}, \ldots, \pp_{k}$ and suppose for some $x$ we have $x\in \pp_{i}$ for $1\leq i\leq k$. Then, we can write $x=\frac{a_{i}}{q_{i}}$ with $a_{i}\in I$ and
$q_{i}\notin \pp_{i}$ for each $i$. Choose, by the Chinese Remainder Theorem, $z_{i}$ for $1\leq i\leq k$ such that $z_{i}\in \pp_{j}$ for $j\neq i$ and $z_{i}\notin \pp_{i}$, and let 
\[a=\sum_{i=1}^k a_{i}z_{i} \mbox{ and  } q=\sum_{i=1}^k q_{i}z_{i}.\] 
Notice that $a\in I$ since $a_{i}\in I$ and $I$ is a fractional ideal. Also, $q\notin \pp_{i}$ since $q_{j}z_{j}\in \pp_{i}$ and therefore $q\equiv z_{i}q_{i} \, \, \textup{mod }\pp_{i}$ which is not 
in $\pp_{i}$ because $z_{i}$, $q_{i}\notin \pp_{i}$ and $\pp_{i}$ is a prime ideal. Then we have
\[x=\frac{\sum_{i=1}^k xq_{i}z_{i}}{\sum_{i=1}^k q_{i}z_{i}}=\frac{a}{q}=\frac{\sum_{i=1}^k a_{i}z_{i}}{\sum_{i=1}^k q_{i}z_{i}}=\frac{a\prod_{\sigma \neq \Id} \sigma(q)}{\prod_{\sigma} \sigma(q)}. \; \; (*)\]

The last product runs over the complex embenddings of $\Q[\beta]$. Now we know that $\prod_{\sigma} \sigma(q)=N(q)\in \Z$, and we will prove that $p\nmid N(q)$. Suppose not and $\prod_{\sigma} \sigma(q)\in (p)\subseteq \pp$ (where $\pp$ is any prime ideal containing $p$); then, since $\pp$ is a prime ideal, $\sigma(q)\in \pp$ for some $\sigma$, and therefore $q\in \sigma^{-1}(\pp)$. But $\sigma^{-1}(\pp)$ is clearly also a prime ideal (because it is maximal); as $p\in \pp$ and $\sigma(p)=p$, $p\in \sigma^{-1}(\pp)$, and therefore $p\in \sigma^{-1}(\pp)=\pp_{i}$ for some $i$, contradicting our construction of $q$. Now by $(*)$ it is clear that $x\in \Z_{(p)}I$, proving that $\bigcap_{\pp \supseteq (p)} I_{\pp}\subseteq \Z_{(p)}I$. \qedhere
\end{proof}

\begin{theorem}
Let $R$ be an order of $K$ containing $\Z[\beta]$ and $I, J$ be $R$-fractional ideals. Then $\Z_{(p)}I$ and $\Z_{(p)}J$ are arithmetically equivalent if and only if $I_{\pp}$ is arithmetically equivalent to $J_{\pp}$ for every prime $\pp\supseteq (p)$.
\label{main}
\end{theorem}

\begin{proof}
For the direct implication, suppose that $\Z_{(p)}I=a\Z_{(p)}J$ for some $a\in \Q[\beta]$. Using the lemma, we get that
\[\bigcap_{\pp\supseteq (p)} I_{\pp}=\Z_{(p)}I=a\Z_{(p)}J=a\bigcap_{\pp\supseteq (p)} J_{\pp}=\bigcap_{\pp\supseteq (p)} aJ_{\pp}=\bigcap_{\pp\supseteq (p)} (aJ)_{\pp}.\]
Localizing this equality in each prime $\pp_{i}\supseteq (p)$, we get that
\[\bigcap_{\pp\supseteq (p)} (I_{\pp})_{\pp_{i}}=\bigcap_{\pp\supseteq (p)} ((aJ)_{\pp})_{\pp_{i}}.\]
But, since $I_{\pp_{i}}\subseteq (I_{\pp})_{\pp_{i}}$ and $I_{\pp_{i}}=(I_{\pp_{i}})_{\pp_{i}}$, the left hand side must be $I_{\pp_{i}}$, and in the same way the right hand side is $aJ_{\pp_{i}}$. Therefore $I_{\pp_{i}}=aJ_{\pp_{i}}$ and $I_{\pp_{i}}$ is arithmetically equivalent to $J_{\pp_{i}}$. As $\pp_{i}$ was arbitrary, we have the desired implication.

Suppose now that $I_{\pp_{i}}$ is arithmetically equivalent to $J_{\pp_{i}}$ for each $1\leq i\leq k$; for each $i$ there exists $a_{i}\in \Q[\beta]$ such that $I_{\pp_{i}}=a_{i}J_{\pp_{i}}$. We will find an $a$ that satisfies this uniformly for every $i$. Every ideal contains a product of some prime ideals, say $(a_{i})\supseteq \pp_{i}^{k_{i}}\prod_{j=1}^{s} \qq_{j}$ where $\qq_{j}\neq \pp_i$ are prime ideals. Notice that $(\qq_j)_{\pp_i}=R_{\pp_i}$ because $\qq_j$ is coprime with $\pp_i$ and $\pp_i R_{\pp_i}$ is the only maximal ideal in $R_{\pp_i}$. Thus, when we localize the equation we get that 
$a_iR_{\pp_i}=(a_i)_{\pp_i}\supseteq (\pp_i^{k_i})_{\pp_i}=\pp_i^{k_i}R_{\pp_i}$.

Consider now $a\equiv a_i \, \, \textup{mod } \pp_i^{k_i}$, which exists by the Chinese Remainder Theorem. Then $a=a_i+p_i$ for some $p_i\in \pp_i^{k_i}$, and thus  $aJ_{\pp_i}=a_iJ_{\pp_i}+p_iJ_{\pp_i}=a_iJ_{\pp_i}=I_{\pp_i}$ where the second equality follows from $p_iJ_{\pp_i}\subseteq \pp_{i}^{k_i}R_{\pp_i}J\subseteq a_iR_{\pp_i}J=a_iJ_{\pp_i}$. We obtain the reverse implication by simply applying the lemma:

\[\Z_{(p)}I=\bigcap_{\pp\supseteq (p)} I_{\pp}=\bigcap_{\pp\supseteq (p)} aJ_{\pp}=a\bigcap_{\pp\supseteq (p)} J_{\pp}=a\Z_{(p)}J.\qedhere\]
\end{proof}

With this result in hands, a criterion to local similarity in terms of ideals is easy because when we apply the Latimer MacDuffee theorem over $\Z_{(p)}$, a PID,  the ideal we get is precisely

\[\Z_{(p)}v_1+\ldots+\Z_{(p)}v_n=\Z_{(p)}I_A.\]

\begin{corollary}
Two matrices $A,B$ are similar over $\Z_p$ if and only if the ideals $(I_A)_{\pp}$ and $(I_B)_{\pp}$ are arithmetically equivalent for every prime ideal $\pp\supseteq (p)$.
\label{maincorol1}
\end{corollary}

\begin{proof}
By Proposition \ref{equiv} matrices $A,B$ are similar over $\Z_p$ if and only if they are similar over $\Z_{(p)}$, and by Latimer-MacDuffee applied to $D=\Z_{(p)}$ 
this happens if and only if $\Z_{(p)}I_A$ and $\Z_{(p)}I_B$ are arithmetically equivalent, which proves the desired result by Theorem \ref{main}. \qedhere
\end{proof}

For this to happen for every $p$ the localizations of the ideals must be arithmetic equivalent for every $\pp$, that is, the ideals must be weakly equivalent.

\begin{corollary}
\label{maincorol2}
Two matrices $A,B$ are similar over $\Z_p$ for every rational prime $p$ if and only if the ideals $I_A$ and $I_B$ are weakly equivalent.
\end{corollary}

\begin{proof}
By Corollary \ref{maincorol1}, $A$ and $B$ are similar over $\Z_p$ for every rational prime $p$ if and only if $(I_A)_{\pp}$ and $(I_B)_{\pp}$ are arithmetically equivalent for every $\pp$, which happens if and  only if $I_A$ and $I_B$ are weakly equivalent, by Proposition \ref{localarithmetic}.\qedhere
\end{proof}

The results obtained so far give rise to two natural questions: on one hand, how to decide about local conjugacy of matrices by the solution of a finite set of modular equations; and on the other, given a pair of matrices, for which primes is local conjugacy not trivial. This last problem will be considered in the next section and we end this one with some remarks concerning the first question and some related issues.

Proposition \ref{equiv} contains already an answer to this question: given matrices $A$ and $B$ and a prime $p$, conjugacy over $\Z_{p}$ is equivalent to conjugacy over $\Z_{/p^{\mu +1}}$.  The computation of the $p$ factor of the Smith Normal Form of an integer matrix may be performed in an efficient way, without the use of determinantal divisors, as shown in \cite{Ger77}.

In dimension $2$ it is possible to give simple criteria for local conjugacy.

\begin{theorem}
Let $A\in M_2(\Z)$ be a matrix with characteristic polynomial $f$ and $\ell=\textup{sup} \{k: \exists \lambda \owns A \equiv \lambda I \, \textup{mod} \ p^k\}$. Then $f$ and $\ell$ completely define
the similarity class of $A$.
\end{theorem}

\begin{proof}
See \cite{AOPV08}.\qedhere
\end{proof}

In the same paper there is a similar, although much more involved, answer to the local similarity problem when $n=3$.

We have also two results telling when is a matrix locally similar to its companion matrix. Notice that by our earlier results this is the same as asking when is $I_A$ invertible.

\begin{lemma}
Suppose $A\in M_n(\Z)$ has characteristic polynomial $f$ and let $C$ be the companion matrix of $f$. Then $A$ is similar to $C$ over $\Z_p$ if and only if (their reductions module $p$) are similar
over $\Z/p\Z$.
\end{lemma}

\begin{proof}
If $A$ is similar to $C$ over $\Z_p$, then it's clear it is also over $\Z/p\Z$ by projecting. A matrix $A$ is similar to its companion matrix if and only if there is a basis $\{v_1, \ldots, v_n\}$ such that $Av_{i}=v_{i+1}$ for $1\leq i<n$, which happens if and only if there is a vector $v=v_1$ such that $\{v, Av, A^2v,\ldots, A^{n-1}v\}$ is a basis. If there is such a vector $v\in \Z/p\Z$, pick any $v'\equiv_p v$ with $v'\in \Z_p$; then $\det(v' , Av', A^2v',\ldots, A^{n-1}v')\equiv_p \det(v, Av, A^2v,\ldots, A^{n-1}v)$ which is nonzero in $\Z/p\Z$, and therefore is a unit over $\Z_p$.
\qedhere
\end{proof}

\begin{corollary}
Suppose $A\in M_n(\Z)$ has characteristic polynomial $f$ such that $f$ doesn't have repeated roots in $\Z/p\Z$. Then $A$ is similar to the companion matrix of $f$ over $\Z_p$.
\end{corollary}

\begin{proof}
Since $A$ has a characteristic polynomial $f$ without repeated roots in $\Z/p\Z$, its characteristic and minimal polynomials over $\Z/p\Z$ are the same, so by the proposition $A$ is similar to its companion matrix over $\Z/p\Z$, and therefore over $\Z_p$.\qedhere
\end{proof}

\section{Orders and Localization}

In this section we consider the second question mentioned in Section 3, namely, the determination of the primes for which local conjugacy is not trivial. We must make this more precise: Fixing a characteristic polynomial we obtain, as described in the introduction, a ring $\Z[\beta]$ with a finite lattice of orders above it. We want to show how local conjugacy of two matrices with associated ideals in the same order depends on the order. In particular, we verify that there exists a finite set of primes, depending only on $\Z[\beta]$, such that all matrices with the given characteristic polynomial are conjugate over $\Z_{p}$, for all $p$ not belonging to that set.

We start by recalling the relations between the ideals in orders as we move ``up'' and ``down'' in the order lattice.

\begin{definition}
Given two commutative rings $S\subseteq R$ let $u:\Id(S)\to \Id(R)$ and $d:\Id(R)\to \Id(S)$ be defined by $u(I)=RI$ and $d(J)=J\cap S$.
\end{definition}

This functions respect some simple properties:

\begin{proposition}
Given rings $S\subseteq R$, we have:
\begin{itemize}
\item $R(J\cap S)=u(d(J))\subseteq J$;
\item $IR\cap S=d(u(I))\supseteq I$;
\item $u(d(u(I)))=u(I)$;
\item $d(u(d(J)))=d(J).$
\end{itemize}
\end{proposition}

This proposition shows that the functions $u$ and $d$ are close to be inverses of each other. In fact this is true in ``a lot of cases.''

\begin{definition}
Given an order $R$ and a rational prime $p$, define $$\Id_p(R)=\{I\in \Id(R): [R:I] \textup{ is a power of } p\}.$$
\end{definition}

Other useful characterizations of $\Id_p(R)$ are the following:

\begin{lemma}
Given an order $R$ and a rational prime $p$, an ideal $I$ is in  $\Id_p(R)$ if and only if there exists $k\in \Z_{\geq 0}$ such that $p^k R\subseteq I\subseteq R$, if and only if $I \cap \Z = p^{k} \Z$. 
\end{lemma}

\begin{proof}
We prove only the equivalence of the first property with one in the definition: the if part is easy since $p^k R\subseteq I\subseteq R$ implies $[R:I]|[R:p^k R]=p^{k n}$. For the only if part, we know that there is some $m$ such that $m R\subseteq I$. Pick a minimal one and suppose that $m$ is not a power of $p$, and let $q\neq p$ be a prime divisor of $m$ and $m'=m/q$. Then $m' R\nsubseteq I$ and $qm'R\subseteq I$; therefore, there is some $x\in m'R$ such that
$x\notin I$ but $qx\in I$, which means that the element $x+I\in R/I$ has order $q$, contradiction with $|R/I|$ being a power of $p$.\qedhere
\end{proof}

Obviously, the set $P_p(R)$ of prime ideals above $p$ is a subset of ${\rm Id}_{p}(R)$. In fact, if $R$ is for instance a Dedekind domain, ${\rm Id}_{p}(R)$ is just the group generated by those prime ideals.

\begin{proposition}
Given an order $R$ and a rational prime $p$, $\Id_p(R)$ is a sub-monoid of $\Id(R)$. 
\end{proposition}

\begin{proof}
We just need to check that $\Id_p(R)$ is closed under multiplication. Suppose $I, J\in \Id_p(R)$; then by the lemma above there are $k, l$ such that $p^k R\subseteq I\subseteq R$ and $p^lR\subseteq J\subseteq R$, implying that $p^{k+l}R\subseteq IJ\subseteq R$, which using the lemma again gives that $IJ\in \Id_p(R)$.\qedhere
\end{proof}

\begin{proposition}
The function $u$ sends elements of $\Id_p(S)$ to elements of $\Id_p(R)$ and the function $d$ sends elements of $\Id_p(S)$ to elements of $\Id_p(R)$.
\end{proposition}

\begin{proof}
Suppose $I\in \Id_p(S)$; then $p^kS\subseteq I\subseteq S$, and multiplying by $R$ we get $p^kR\subseteq IR \subseteq R$, and therefore $IR\in \Id_p(R)$. If $J\in \Id_p(R)$, $p^kR\subseteq J\subseteq R$, 
and intersecting with $S$ gives $p^kS\subseteq p^kR\cap S\subseteq J\cap S\subseteq S$, meaning that $J\cap S\in \Id_p(S)$.\qedhere
\end{proof}

\begin{lemma}
Let $I\in \Id_p(S)$, $J\in \Id_p(R)$ and suppose that $p\nmid[R:S]$. Then $RI+S=R=J+S$.
\end{lemma}

\begin{proof}
Clearly $IR, S\subseteq IR+S\subseteq R$. Therefore $[R:IR+S][IR+S:S]=[R:S]$ and $[R:IR+S][IR+S:IR]=[R:IR]$, implying that $[R:IR+S]$ divides both $[R:S]$ and $[R:IR]$; but the first one is coprime with $p$ and the second one is a power of $p$, which gives $[R:IR+S]=1$ and $IR+S=R$. The second  equality is analogous. \qedhere
\end{proof}

\begin{theorem}
Let $I\in \Id_p(S)$, $J\in \Id_p(S)$ and suppose that $p\nmid[R:S]$. Then $IR\cap S=I$ and $R(J\cap S)=J$. 
\end{theorem}

\begin{proof}
Notice that $I\subseteq IR\cap S\subseteq S$. Since $IR\cap S\subseteq S$, we know that $[I:
I\subseteq IR\cap S]$ divides $[S:I]$, and therefore $[I:IR\cap S]$ is a power of $p$. On the other hand, by the second isomorphism theorem we have $S/(IR\cap S)\cong (IR+S)/(IR)=R/(IR)$. This gives
$[I:IR\cap S][IR: I]=[IR:IR\cap S]=[R:S]$ which is not divisible by $p$, making $[I:IR\cap S]=1$ and $IR\cap S=I$, proving the first equation.

Notice that $R(J\cap S)\subseteq RJ=J$. Sinde $J\cap S\subseteq R(J\cap S) \subseteq R$, $[J:R(J\cap S)][R:J]=[R:R(J\cap S)]$ which is a power of $p$ since $J\in \Id_p(R)\Rightarrow u(d(J))\in \Id_p(R)$.
On the other hand, $J/(J\cap S)\cong (J+S)/S=R/S$ by the second isomorphism theorem, giving $[J:R(J\cap S)][R(J\cap S): J\cap S]=[J:J\cap S]=[R:S]$, and we can finish as before. \qedhere
\end{proof}

\begin{corollary}
If $p\nmid[R:S]$ the restrictions of $u$ and $d$ are inverse isomorphisms between $\Id_p(S)$ and $\Id_p(R)$.
\end{corollary}

\begin{proof}
By the above theorem they are inverse functions. Since
\[ u(I)u(J)=(RI)(RJ)=RIJ=u(IJ),\]
we get that $u$ is an isomorphism; $d$ is also an isomorphism because inverses of isomorphisms are isomorphisms as well.
\qedhere
\end{proof}

\begin{proposition}
Suppose that $p\nmid [R:S]$ and let $I\in \Id_p(S)$, $J\in \Id_p(R)$ be corresponding ideals by the isomorphisms $u$ and $d$ (that is, $u(I)=J$ and $d(J)=I$). Then $S/I\cong R/J$.
\end{proposition}

\begin{proof}
This was done while proving the preceding theorem.\qedhere
\end{proof}

We now turn our attention to prime ideals. We will see how (in some cases) the functions $u$ and $d$ preserve prime ideals. 

\begin{proposition}
Let $S\subseteq R$ be orders and $\wp$ a prime ideal of $R$. Then $S/(\wp\cap S)\cong R/\wp$ and $\wp\cap S$ is a prime ideal of $S$. 
\end{proposition}

\begin{proof}
Notice that $\wp \subseteq S+\wp \subseteq R$ and $S+\wp \neq \wp$ because that would imply that $S\subseteq \wp$ which can't happen since $1\notin \wp$. As $\wp$ is maximal, this gives $S+\wp =R$ and by the second isomorphism theorem $S/(\wp\cap S)\cong (S+\wp)/\wp= R/\wp$. Now $\wp$ being prime means that $R/\wp$ is a domain and therefore $S/(\wp\cap S)$ is also a domain and $\wp\cap S$ is a prime ideal 
of $S$.\qedhere
\end{proof}

This means that it's always the case that going down preserves prime ideals. This is not true in general for going up, but it is if we assume also that $p\nmid [R:S]$.

\begin{proposition}
Let $S\subseteq R$ be orders and $\pp\in P_p(S)$. Suppose that $p\nmid [R:S]$. Then $S/\pp\cong R/(R\pp)$ and $R\pp$ is a prime ideal of $R$.
\end{proposition}

\begin{proof} 
Since $\pp\in \Id_p(S)$, $S/\pp\cong R/(R\pp)$. Since $\pp$ is prime, and therefore maximal, the left hand side is a field, and so the right hand side is a field as well implying $R\pp$ is a maximal ideal 
of $R$.\qedhere
\end{proof}

Since $u$ and $d$ are isomorphism (preserving the multiplicative structure) factorizations into prime ideals are also preserved. So let us now apply these results when $R=O_K$ is a Dedekind domain.

\begin{theorem}
Let $S$ be an order with fraction field $K$ and $p$ be a prime such that $p\nmid [O_K:S]$. Then the ideal $pS$ has a factorization in prime ideals of $S$. Also, if we factor
$pO_K=\prod_{i=1}^m \wp_i^{k_i}$ we can factor $pS=\prod_{i=1}^m \pp_i^{k_i}$ where $\pp_i=d(\wp_i)$.
\end{theorem}

\begin{proof}
The factorization of $pO_k$ always exists because $O_k$ is a Dedekind domain, and notice that clearly $\wp_i\in P_p(O_K)\subseteq \Id_P(O_K)$. Now just apply $d$ to the factorization and use the fact that it is an isomorphism. \qedhere
\end{proof}

\begin{corollary}
Let $S$ be a order with fraction field $K$ and $p$ be a prime such that $p\nmid [O_K:S]$. Then the order of an ideal $\pp\in P_p(S)$ is always $S$ and $\pp$ is invertible in $S$.
\end{corollary}

If $\pp \triangleleft S$ is a prime ideal of $S$, then $\pp$ is invertible if and only if $(\pp:\pp)=S$.

We will now prove a converse of the above corollary showing that the rational primes $p$ under the prime ideals that fail to be invertible are precisely the ones dividing $[O_K:S]$.

\begin{proposition}
Let $S$ be a order with fraction field $K$ and $p$ be a prime. Every prime $\pp\in P_p(S)$ is invertible if and only if $p\nmid [O_K:S]$.
\end{proposition}

\begin{proof}
The if part is already done. Suppose now that $p|[O_K:S]$ and every prime ideal $\pp\in P_p(S)$ is invertible. Therefore $\Z_{(p)}\pp\triangleleft \Z_{(p)}S$ are invertible and, since they are all the maximal ideals in $\Z_{(p)}S$, it follows that every ideal in $\Z_{(p)}S$ is invertible. But since $\Z_{(p)}S$ is a semi-local ring, by Proposition \ref{semilocal}, it follows that $\Z_{(p)}S$ is a PID. As $\Z_{(p)}S \, \Z_{(p)}O_K=\Z_{(p)}O_K$ we know that $\Z_{(p)}O_K$ is a $\Z_{(p)}S$-module, and therefore it is a fractional ideal (it's clearly finite). Then $\Z_{(p)}O_K=a\Z_{(p)}S$; however,
$$\Z_{(p)}O_K=\Z_{(p)}S \, \Z_{(p)}O_K=\frac{1}{a}\Z_{(p)}O_K \, \Z_{(p)}O_K=\frac{1}{a}\Z_{(p)}O_K=\Z_{(p)}S.$$
Assuming $p||O_K/S|$, by Cauchy Theorem there is some $x\in O_K \setminus S$ such that $px\in S$. For this $x$, $x\in O_K\subseteq \Z_{(p)}O_K=\Z_{(p)}S$, so we can write $x=\frac{s}{q}$ where $s\in S$ and $q\in \Z\setminus p\Z$. But this means that $ps=qpx\in qS\cap pS=pqS$ (because $pS$ and $qS$ are coprime), and therefore $x=\frac{ps}{pq}\in S$, contradiction! \qedhere
\end{proof}

Recall that, by Corollary \ref{maincorol1} from Section 3, two matrices $A$ and $B$ are conjugate over $\Z_{p}$ if and only if the localizations of the ideals $I_{A}$ and $I_{B}$ at every prime ideal $\pp$  above $p$ are arithmetically equivalent. But, if $p\nmid [O_K:S]$, where $S$ is the ring of coefficients of those ideals, the $\pp$ above $p$ are all invertible and so $S_{\pp}$ are PID and all ideals are arithmetically equivalent.

In particular, if $f(x)$ is the minimal polynomial of $\beta$ and $p\nmid [O_{K}:\Z[\beta]]$, all matrices with characteristic polynomial $f(x)$ are conjugate over $\Z_{p}$. A sufficient condition for this is that $p^{2}$ doesn't divide the discriminant of $f(x)$.

\section{Concluding remarks}
As already mentioned in the introduction,  it was proved in \cite{BR12} that given $A,B\in GL(n,\Z)$ with the same irreducible characteristic polynomial, weak equivalence of the associated ideals $I_{A}$ and $I_{B}$ implies that $A$ and $B$ are conjugate over $\overline{\Z}$, the profinite completion of $\Z$, which is nothing but the direct product of the $p$-adic rings $\Z_{p}$. The results in the present paper provide a converse and an alternative proof of that statement.

On the other hand, in \cite{BR15} it was proved that $I_{A}$ and $I_{B}$ are weakly equivalent if and only if $A$ and $B$ are $2$-block conjugate, i.e., there exist matrices $A',B'\in GL(n,\Z)$ and $M,N \in GL(2 n,\Z)$ such that \[(A \oplus A')M=M(B \oplus B) \,\,\mbox{  and  } \,\,  (A \oplus A)N=N(B \oplus B').\]
When $A$ and $B$ are not conjugate over $\Z$, a $2$-block conjugacy of $A$ and $B$ intertwines the dynamics of $A$ and $B$ in a nontrivial manner. Although some of the dynamical consequences of $2$-block conjugacy of $A$ and $B$ have been explored in \cite{BR15}, there remains much to be understood.

A fruitful stage for the study of the conjugacy of integer matrices is provided by its interplay among weak equivalence of ideals, $2$-block conjugacy, and conjugacy over $\Z_p$. As seen here, different aspects of the conjugacy problem of integer matrices are manifested in several terms that combine various algebraic and dynamical properties associated to the integer matrices. The hope is that from this stage, new invariants for the conjugacy problem for irreducible hyperbolic toral automorphisms will be found for arbitrary $n$.

\end{document}